\documentclass{amsart}

\usepackage{amsfonts}
\usepackage{amsmath}
\usepackage{amsrefs}
\usepackage{amssymb}
\usepackage{amsthm}
\usepackage[english]{babel}
\usepackage{color}
\usepackage{epsfig}
\usepackage{graphicx}
\usepackage{hyperref}
\usepackage[utf8]{inputenc}
\usepackage{nicefrac}
\usepackage{psfrag}
\usepackage{xcolor}

\hypersetup{
    colorlinks,
    citecolor=blue,
    filecolor=blue,
    linkcolor=blue,
    urlcolor=blue}

\theoremstyle{plain}
\newtheorem{thm}{Theorem}[section]

\newtheorem{lem}[thm]{Lemma}
\newtheorem{prop}[thm]{Proposition}

\theoremstyle{definition}

\newtheorem{df}[thm]{Definition}

\theoremstyle{remark}

\newcommand{\Z}{\mathbb Z}
\newcommand{\N}{\mathbb N}

\begin{document}

\author{M. Achigar}
\title{A note on Anosov homeomorphisms}
\date{\today}
\maketitle

\begin{abstract}
For an $\alpha$-expansive homeomorphism of a compact space we give an elementary proof of the following well-known result in topological dynamics: A sufficient condition for the homeomorphism to have the shadowing property is that it has the $\alpha$-shadowing property for one-jump pseudo orbits (known as the local product structure property). The proof relies on a reformulation of the property of expansiveness in terms of the pseudo orbits of the system.
\end{abstract}

\section{Introduction}

In \cite{AH}*{Theorem 1.2.1} it is proved, among other things, that \emph{Anosov diffeomorphisms} has the \emph{shadowing property}, called \emph{pseudo orbit tracing property}\footnote{For terminology and notation in the Introduction we refer \cite{AH}*{\S 1.2}} there. In the proof, on \cite{AH}*{p. 23}, the authors only uses the so called \emph{local product structure property}: if $d(x,y)<\delta$ then $W^s_\varepsilon(x)\cap W^u_\varepsilon(y)\neq\varnothing$ if $\delta>0$ is chosen small enough for a given $\varepsilon>0$, and the special (hyperbolic) properties of the metric $d$ coming from the Riemannian structure of the manifold supporting the system  \cite{AH}*{(B), p.\,20}.

As can be easily checked the first of these two conditions is equivalent to the shadowing property for pseudo orbits with one jump, that is, for every $\varepsilon>0$ there exists $\delta>0$ such that for every bi-sequence of points of the form 
$$
\ldots,\;z_{-2}=T^{-2}y,\;z_{-1}=T^{-1}y,\;z_0=x,\;z_1=Tx,\;z_2=T^2x,\;\ldots
$$
with $d(x,y)<\delta$, where $T$ denotes the diffeomorphism, there exists a point $z$ such that $d(T^nz,z_n)<\varepsilon$ for all $n\in\Z$.

On the other hand, in \cite{Fat}*{Theorem 5.1} it is shown that for every expansive homeomorphism on a compact space there exists a compatible metric (which we call \emph{hyperbolic metric}) with similar properties to those of the metric $d$ in the case of Anosov diffeomorphisms. Then the proof of the shadowing property for Anosov diffeomorphisms given in \cite{AH} carries over the more general case of expansive systems.

In this note we give an alternative and elementary proof of this well-known shadowing condition (Proposition \ref{prop:1step_shadowing}) not making use of Fathi's hyperbolic metric. Instead we use a reformulation of the property of expansiveness of a system (Proposition \ref{prop:po_semiexp}) which seems interesting on its own right.

\section{Terminology and notation}

In this note $X$ denotes a compact metric space with metric $d$ and $T\colon X\to X$ a homeomorphism. The \emph{orbit} of a point $x\in X$ is the bi-sequence $O(x)=(T^nx)_{n\in\Z}$.

\begin{df}\label{def:exp}
 $T$ is said to be \emph{expansive} if there is $\alpha>0$ (called  \emph{expansivity constant}) such that if $x,y\in X$ and $d(T^nx,T^ny)\leq\alpha$ for all $n\in\Z$ then $x=y$.
\end{df}

\begin{df}
 Let $\xi=(x_n)_{n\in\Z}$ be a bi-sequence of elements of $X$.  If $\delta>0$ and $d(Tx_n,x_{n+1})<\delta$ for all $n\in\Z$ then $\xi$ is called \emph{$\delta$-pseudo orbit}. We say that $\xi$ has a \emph{jump} at the $n$-th step if $Tx_{n-1}\neq x_n$. Given $\varepsilon>0$ a bi-sequence $\eta=(y_n)_{n\in\Z}$ is said to \emph{$\varepsilon$-shadow} $\xi$ if 
 $d(x_n,y_n)<\varepsilon$ for all $n\in\Z$. If in the previous situation $\eta=O(x)$ is the orbit of a point $x\in X$ we simply say that $x$ $\varepsilon$-shadows $\xi$ and that $\xi$ is $\varepsilon$-shadowed (or $\varepsilon$-shadowable).
\end{df}

\begin{df}
 Given $\varepsilon>0$ we say that $T$ has the \emph{$\varepsilon$-shadowing property} if for some $\delta>0$ every $\delta$-pseudo orbit is $\varepsilon$-shadowable. We say that $T$ has the \emph{shadowing property} if it has the $\varepsilon$-shadowing property for all $\varepsilon>0$. If $T$ is expansive and has the shadowing property then it is called \emph{Anosov homeomorphism}.
\end{df}

\section{Rephrasing expansivity}

The following simple result states an equivalent condition for the expansiveness of the system $(X,T)$. This alternative characterization of expansiveness will allow us to give an elementary proof of the shadowing condition in Proposition \ref{prop:1step_shadowing}.

\begin{prop}\label{prop:po_semiexp}
 Let $\alpha>0$. The following conditions are equivalent.
 \begin{enumerate}
  \item $T$ is expansive with expansivity constant $\alpha$.
  \item For every $\varepsilon>0$ there exists $\delta>0$ such that
  $$
  \qquad\text{if}\quad d(x_n,y_n)\leq\alpha\text{ for all }n\in\Z\quad\text{then}\quad d(x_n,y_n)<\varepsilon\text{ for all }n\in\Z,
  $$
  for every pair of $\delta$-pseudo orbits $(x_n)_{n\in\Z}$ and $(y_n)_{n\in\Z}$ of $T$.
 \end{enumerate}
\end{prop}

\begin{proof}
 \mbox{$(1\Rightarrow2)$\;} Suppose that the thesis is not true. Then, there exists $\varepsilon>0$ such that for every $k\in\N$ one can find $\nicefrac1k$-pseudo orbits $(x^k_n)_{n\in\Z}$ and $(y^k_n)_{n\in\Z}$ of $T$ satisfying $d(x_n^k,y_n^k)\leq\alpha$ for all $n\in\Z$ but $d(x_{n_k}^k,y_{n_k}^k)\geq\varepsilon$ for a suitable $n_k\in\Z$. Changing the indexing of the pseudo orbits if necessary it can be assumed that $n_k=0$ for all $k\in\N$. As $X$ is compact it can be also assumed that $x_0^k\to x$ and $y_0^k\to y$ for some $x,y\in X$. It is easy to see that then $x_n^k\to T^nx$ and $y_n^k\to T^ny$ for all $n\in\Z$ (the pseudo orbits converge pointwise to actual orbits). But now, as $d(x_n^k,y_n^k)\leq\alpha$ for all $k\in\N$ and $n\in\Z$ we have $d(T^nx,T^ny)\leq\alpha$ for all $n\in\Z$, and as $d(x_0^k,y_0^k)\geq\varepsilon$ for al $k\in\N$ we get $d(x,y)\geq\varepsilon$, so that $x\neq y$. This contradicts that $\alpha$ in an expansivity constant and the proof finishes.

 \mbox{$(2\Rightarrow1)$\;} Suppose $x,y\in X$ verifies $d(T^nx.T^ny)\leq\alpha$ for all $n\in\Z$,  and note that $(T^nx)_{n\in\Z}$ and $(T^ny)_{n\in\Z}$ are $\delta$-pseudo orbits for every $\delta>0$. Therefore, by the hypothesis, for every $\varepsilon>0$ we have $d(T^nx,T^ny)<\varepsilon$ for all $n\in\Z$, that is, $(T^nx)_{n\in\Z}=(T^ny)_{n\in\Z}$. Then $x=y$ and hence $\alpha$ is an expansivity constant.
\end{proof}

For later reference we recall from \cite{Br62}*{Theorem 5} the following basic property of expansive homeomorphisms on compact spaces known as \emph{uniform expansivity}.

\begin{prop}\label{prop:uexp}
 Let $\alpha>0$. The following conditions are equivalent.
 \begin{enumerate}
  \item $T$ is expansive with expansivity constant $\alpha$.
  \item For every $\varepsilon>0$ there exists $N\in\N$ such that for every $x,y\in X$
  $$
  \text{if}\quad d(T^nx,T^ny)\leq\alpha\text{ for all }|n|\leq N\quad\text{then}\quad d(x,y)<\varepsilon.
  $$
 \end{enumerate}
\end{prop}

We also recall the following easy result that can be found in  \cite{Wal}*{Lemma 8}.

\begin{lem}\label{lem:finite_shadowing}
 $T$ has the shadowing property if and only if $T$ has the shadowing property for pseudo orbits with a finite number of jumps.
\end{lem}

\section{The shadowing condition}

As pointed out in the Introduction the next is a known result that can be proved with the techniques in \cite{AH}*{p. 23} replacing the metric coming from the Riemannian structure in that argument by Fathi's hyperbolic metric \cite{Fat}*{Theroem 5.1}.

\begin{prop}\label{prop:1step_shadowing}
 If $T$ is expansive with expansivity constant $\alpha>0$ then the following conditions are equivalent.
  \begin{enumerate}
  \item $T$ has the shadowing property.
  \item There exists $\delta>0$ such that every one-jump $\delta$-pseudo orbit is $\alpha$-shadowed.
 \end{enumerate}
\end{prop}

\begin{proof} 
 Clearly we only need to prove that the last statement implies the first one. By Lemma \ref{lem:finite_shadowing} it is enough to show that for every $\varepsilon>0$ there exists $\rho>0$ such that all $\rho$-pseudo orbits with a finite number of jumps are $\varepsilon$-shadowed. To do that it is sufficient to find a $\rho>0$ corresponding only to $\varepsilon=\alpha$, because by Proposition \ref{prop:po_semiexp} for any $\varepsilon>0$ taking a smaller value of $\rho$, more precisely choosing $\rho\leq\delta$ where $\delta$ is given by the cited proposition, we have that to $\alpha$-shadow a $\rho$-pseudo orbit is equivalent to $\varepsilon$-shadow it.

 To find $\rho>0$ such that every $\rho$-pseudo orbit with a finite number of jumps is $\alpha$-shadowed, let $\delta>0$ be as in the statement of this proposition, that is, such that
 \begin{equation}\label{ec:001}
  \textit{every }\delta\textit{-pseudo orbit with one jump is }\alpha\textit{-shadowed.}
 \end{equation}
 By Proposition \ref{prop:po_semiexp} (with $\varepsilon=\nicefrac\alpha2$) we can take a smaller $\delta$ to also guarantee that
 \begin{equation}\label{ec:002}
  \textit{if a }\delta\textit{-pseudo orbit }\xi\textit{ }\alpha\textit{-shadows a }\delta\textit{-pseudo orbit }\eta\textit{ then }\xi\textit{ }\nicefrac\alpha2\textit{-shadows }\eta.
 \end{equation}
 Obviously we can also require that $\delta\leq\alpha$.
 For this $\delta$ there exists $N\in\N$ such that
 \begin{equation}\label{ec:003}
  \textit{if }d(T^nx,T^ny)\leq\alpha\textit{ for all }|n|\leq N\text{ then }d(x,y)<\delta,
 \end{equation}
 for all $x,y\in X$, according to Proposition \ref{prop:uexp}. Finally, as $T$ is uniformly continuous we can take $\rho>0$, $\rho\leq\delta$, such that  any segment of length $2N+1$ of a $\rho$-pseudo orbit, say $x_0,\ldots,x_{2N+1}$, is $\delta$-shadowed by its first element $x_0$, that is,
 \begin{equation}\label{ec:004}
  \textit{if }d(Tx_n,x_{n+1})<\rho,\,0\leq n\leq2N,\,\textit{then }d(T^nx_0,x_n)<\delta,\,0\leq n\leq2N+1.
 \end{equation}

 We will prove that this $\rho$ works by induction in the number of jumps in the $\rho$-pseudo orbits. If a $\rho$-pseudo orbit has only one jump, as $\rho\leq\delta$ we know by condition (\ref{ec:001}) that it can be $\alpha$-shadowed. Assume now that $\xi=(x_n)_{n\in\Z}$ is a $\rho$-pseudo orbit with $k\geq2$ jumps. Indices can be arranged so that the last jump takes place in the step from $x_{2N}$ to $x_{2N+1}$, so that $(x_n)_{n>2N}$ is a segment of a true orbit. By condition (\ref{ec:004}) we can replace $(x_n)_{n=0}^{2N}$ by $(T^nx_0)_{n=0}^{2N}$ in $\xi$ getting a $\delta$-pseudo orbit\footnote{We denote $(x_n)_{n\in I}\sqcup(x_n)_{n\in J}=(x_n)_{n\in I\cup J}$ if $I,J\subseteq\Z$ are disjoint sets of indices.} 
 $$
 \xi'=(x_n)_{n<0}\sqcup(T^nx_0)_{0\leq n\leq 2N}\sqcup(x_n)_{n>2N}
 $$
 which $\alpha$-shadows $\xi$ because $\delta\leq\alpha$. Note that on one hand the bi-sequence given by
 $$
 \eta=(x_n)_{n<0}\sqcup(T^nx_0)_{n\geq0}
 $$
 is a $\rho$-pseudo orbit with less than $k$ jumps, then by the inductive hypothesis there exists $y\in X$ that $\alpha$-shadows $\eta$. By condition (\ref{ec:002}) we know that in fact $\eta$ is $\nicefrac\alpha2$-shadowed by $y$. On the other hand consider 
 $$
 \zeta=(T^nx_0)_{n\le2N}\sqcup(x_n)_{n>2N}
 $$ 
 which is a $\delta$-pseudo orbit with one jump. Then by condition (\ref{ec:001}) there exists $z\in X$ that $\alpha$-shadows $\zeta$. Again condition (\ref{ec:002}) implies that $\zeta$ is $\nicefrac\alpha2$-shadowed by $z$.

 Now, as the segment of orbit $(T^nx_0)_{0\leq n\leq2N}$ is in both sequences $\eta$ and $\zeta$ we have that the corresponding segments of the orbits of $y$ and $z$ verifies $d(T^ny,T^nz)<\alpha$ for $0\leq n\leq2N$. Hence, by condition (\ref{ec:003}) we have that $d(T^Ny,T^Nz)<\delta$. Consequently 
 $$
 \tau=(T^ny)_{n<N}\sqcup(T^nz)_{n\geq N}
 $$
 is a one-jump $\delta$-pseudo orbit that $\nicefrac\alpha2$-shadows $\xi'$. A new application of condition (\ref{ec:001}) gives an element $w\in X$ that $\alpha$-shadows $\tau$. Finally, as $w$ $\alpha$-shadows $\tau$, $\tau$ $\nicefrac\alpha2$-shadows $\xi'$ and $\xi'$ $\alpha$-shadows $\xi$, we obtain by repeated application of condition (\ref{ec:002}) that $w$ $\alpha$-shadows $\xi$, and we are done. 
\end{proof}

\vspace{1cm}

\begin{flushleft}
 {\small {\sc Departamento de Matemática y Estadística del Litoral,\\ Universidad de la Rep\'ublica, Gral. Rivera 1350, Salto, Uruguay.}\\
 Partially supported by SNI--ANII.\\
 Email: machigar@unorte.edu.uy.}
\end{flushleft}

\end{document}